\documentclass[11pt]{amsart}
\usepackage{amscd}
\usepackage{amsxtra}
\usepackage{amssymb, amsmath}
\usepackage{amsthm}
\usepackage{pifont}
\usepackage[mathscr]{eucal}

\theoremstyle{plain}
\newtheorem{theorem}{Theorem}[subsection]
\newtheorem{proposition}[theorem]{Proposition}
\newtheorem{lemma}[theorem]{Lemma}
\newtheorem{corollary}[theorem]{Corollary}
\newtheorem{definition}[theorem]{Definition}
\newtheorem{remark}[theorem]{Remark}

\begin{document}

\newcommand{\U}{\ensuremath{\mathfrak{u}_{1}}}

\newcommand{\ad}{\ensuremath { ad(\mathfrak{u}_{1})}}
\newcommand{\Oo}{\ensuremath{\varOmega^{0}(ad(\mathfrak{u}_{1}))}}
\newcommand{\Ok}{\ensuremath {\varOmega^{1}(ad(\mathfrak{u}_{1}))}}

\newcommand{\svb}{\ensuremath{\mathcal{S}}}
\newcommand{\svbp}{\ensuremath{\mathcal{S}^{+}}}
\newcommand{\svbn}{\ensuremath{\mathcal{S}^{-}}}
\newcommand{\csa }{\ensuremath {\mathcal{S}_{\mathfrak{c}}}}
\newcommand{\cs}{\ensuremath{\mathcal{S^{\mathfrak{c}}}}}
\newcommand{\vsa}{\ensuremath{\Omega^{0}(\mathcal{S}^{+}_{\mathfrak{c}})}}
\newcommand{\vsan}{\ensuremath{\Gamma(\mathcal{S^{-}_{\mathfrak{c}}})}}
\newcommand{\la }{\ensuremath{{\mathcal{L}}_{\mathfrak{c}}}}
\newcommand{\sla }{\ensuremath{{\mathcal{L}}^{1/2}_{\mathfrak{c}}}}
\newcommand{\pcsa }{\ensuremath {\mathcal{S}^{+}_{\mathfrak{c}}}}

\newcommand{\ca}{\ensuremath{\mathcal{C}_{\mathfrak{c}}}}
\newcommand{\cas}{\ensuremath{\mathcal{C}^{\sigma}_{\mathfrak{c}}}}
\newcommand{\Aa}{\ensuremath {\mathcal{A}_{\mathfrak{c}}}}
\newcommand{\Q}{\ensuremath{\mathcal{A}_{\mathfrak{c}}\times_\mathcal{G_{\mathfrak{c}}}
\varGamma (S^{+}_{\mathfrak{c}})}}
\newcommand{\B}{\ensuremath{\mathcal{B}_{\mathfrak{c}}}}
\newcommand{\Bs}{\ensuremath{\mathcal{B}^{\sigma}_{\mathfrak{c}}}}
\newcommand{\Bn}{\ensuremath{\mathcal{B}^{N}_{\mathfrak{c}}}}
\newcommand{\jx}{\ensuremath{\mathcal{J}_{M}}}

\newcommand{\G}{\ensuremath{\mathcal{G}}}

\newcommand{\h}{\ensuremath{H^{\mathcal{S}\mathcal{W}}_{(A,\phi)}}}
\newcommand{\hh}{\ensuremath{\widehat{H}^{\mathcal{S}\mathcal{W}}_{(A,\phi)}}}

\newcommand{\lda}{\lambda}

\newcommand{\s}{\ensuremath{\mathfrak{s}}} 
\newcommand{\cc}{\ensuremath{\mathfrak{c}}} 
\newcommand{\spinc}{\ensuremath{spin^{\mathfrak{c}}}}
\newcommand{\cl}{\boldsymbol{\mathsf{c}}} 

\newcommand{\Z }{\ensuremath{\mathbb {Z}}}
\newcommand{\R }{\ensuremath{\mathbb {R}}}
\newcommand{\C }{\ensuremath {\mathbb {C}}}
\newcommand{\N }{\ensuremath {\mathbb {N}}}

\newcommand{\iso }{\ensuremath {\thickapprox }}

\newcommand{\sw}{\ensuremath {\mathcal{S}\mathcal{W}}}
\newcommand{\swa}{\ensuremath {\mathcal{S}\mathcal{W}_{\mathfrak{c}}}}
\newcommand{\swh}{\ensuremath {\mathcal{S}\mathcal{W}^{h}_{\mathfrak{c}}}}
\newcommand{\swn}{\ensuremath {\mathcal{S}\mathcal{W}^{\eta}_{\mathfrak{c}}}}
\newcommand{\sws}{\ensuremath {\mathcal{S}\mathcal{W}^{\sigma}_{\mathfrak{c}}}}
\newcommand{\y}{\ensuremath {\mathcal{Y}\mathcal{M}}}
\newcommand{\yp}{\ensuremath {\mathcal{Y}\mathcal{M}^{+}}}
\newcommand{\yn}{\ensuremath {\mathcal{Y}\mathcal{M}^{-}}}
\newcommand{\Cf }{\ensuremath {\mathcal{C}_{\alpha}}}
\newcommand{\w }{\ensuremath {\omega}}
\newcommand{\cx}{\ensuremath {C_{X}}}

\title{Instability of Reducible Critical Points of the Seiberg-Witten Functional\footnote{research partially supported by FAPESC 2568/2010-2}}
\author{ Celso M. Doria \\ Depto. de Matem\'atica, UFSC}
\maketitle
\begin{abstract}
\noindent The Euler-Lagrange equations for the variational approach to the Seiberg-Witten equations always admit reducible solutions. In this context, the existence of unstable  reducible solutions is achieved by  assuming the existence of a parallel spinor or the negativeness of a Perelman-Yamabe type of invariant defined for  a $\spinc$-structure. 
\end{abstract}

\section{\bf{Introduction}}

Let  $(M,g)$ be a closed riemannian four manifold with scalar curvature $k_{g}$. By considering the least eigenvalue $\lambda_{g}$ of the operator $\triangle_{g}+\frac{k_{g}}{4}$, where 
$\triangle_{g}=d^{\ast}d$ is the Laplace-Beltrami operator associated to $g$,  Perelman introduced in  ~\cite{Pe01,Pe02}  the smooth invariant 

\begin{equation}
\bar{\lambda}(M)=sup_{g\in\mathscr{M}}\lambda_{g}[vol(M,g)]^{1/2}
\end{equation}

\noindent where $\mathscr{M}$ is the space of $C^{\infty}$-metrics  on $M$. Let $[g]$ be the conformal class of $g\in\mathscr{M}$,  Kobayashi ~\cite{Kob} and Schoen ~\cite{RS87} independently introduced the smooth manifold invariant 

\begin{equation}
\mathcal{Y}(M)=\sup_{[g]}\inf_{g}\frac{\int_{M}k_{g}dv_{g}}{vol(M,g)^{1/2}} .
\end{equation}

\noindent Assuming $\mathcal{Y}(M)<0$,  Akutagawa-Ishida-Le Brun proved  ~\cite{AIL} the equality $\bar{\lambda}(M)=\mathcal{Y}(M)$. A similar quantity turns up by measuring the instability of reducible critical points of 
the Seiberrg-Witten  functional, though in this case it depends on a $\spinc$-structure on $M$.   
There exist smooth 4-manifolds admitting a $\spinc$ structure $\cc$ such that the Seiberg-Witten invariant $SW(\cc)\ne 0$.  
These $\spinc$ structures are named basic classes and they are in the realm of the 4-dim differential topology.  The space $Spin^{\cc}(M)$ of $\spinc$ structures on $M$ might be identified with

\begin{equation}\label{E:00}
\{ \cc=\alpha_{\cc}+\beta_{\cc} \in H^{2}(X,\Z)\oplus H^{1}(X,\Z_{2}) \mid w_{2}(X)=
\alpha \ mod\ \ 2\}.
\end{equation}

\vspace{05pt}

\noindent  From the analytical point of view, a basic classes carries a $\swa$-monopole, which is a special solution of a partial differential equation, as it will be defined next. The motivation for this research was to use variational techniques to  measure the instability of  reducible critical points for the Seiberg-Witten functional; the space of reducible solutions is diffeomorphic to the jacobian torus $\jx=\frac{H^{1}(M,\R)}{H^{1}(M,\Z)}$. 
Based on the fact that the Seiberg-Witten invariants are also  expectation values of a $N=4$ supersymmetric  twisted gauge theory,  ~\cite{LaLo},
one might believe that either there exists a monopole or $\jx$ achieves the minimum energy.  

 The isomorphisms  $Spin^{c}_{4}=(SU_{2}\times SU_{2}\times U_{1})/ _{\Z_{2}}$  and  $Spin^{c}_{3}=U_{2}=(SU_{2}\times U_{1})/ \Z_{2}$ induce the 
 representations  $\rho_{\pm}:Spin_{4} \rightarrow U_{2}=(SU_{2}\times U_{1})/ \Z_{2}$.  Let $P$ be the $\spinc$-principal bundle over $M$ induced by the class $\cc\in Spin^{c}(M)$ and 
 $\csa^{\pm}=P\times_{\rho_{\pm}}\C^{2}$. In practice, a $\spinc$-structure on $M$ is means the existence of  a pair of rank 2 complex 
 vector bundles $\csa^{\pm}$, which fibers are $Spin^{c}_{4}$-modules,  and isomorphisms $det(\csa^{+})=det(\csa^{-})=\la$, where $det(\csa^{\pm})$ are the determinant line bundle such that $c_{1}(\la)=\alpha_{\cc}\in H^{2}(M,\Z)$. 
 Let $\vsa$ the space of sections on $\csa^{+}$ and $\Aa$  be the space of $U_{1}$-connections 1-forms. Each $A\in \Aa$ induces a covariant derivative $\nabla^{A}:\Omega^{0}(\la)\rightarrow\Omega^{1}(\la)$ on $\la$. 
  E.Witten introduced in  ~\cite{Wi94} the  coupled system of $1^{st}$-order PDE ($SW$-monopole eqs.),  

\begin{equation}\label{E:01}
\begin{aligned}
 D^{+}_{A}\phi=0,\quad &(2.1)\\
 F^{+}_{A}=\sigma(\phi),\quad &(2.2),
\end{aligned}
\end{equation}

\vspace{05pt}

\noindent  where   $\phi\in\vsa$, $D^{+}_{A}$ is the positive component of the Dirac
operator, $F^{+}_{A}$ is the self-dual component of the curvature
$F_{A}$ and $\sigma:\vsa\rightarrow\Omega^{2}_{+}(i\R)$ is the self-dual 2-form 

\begin{equation*}
\sigma(v)(X,Y)=<X.Y.v,v>+\frac{1}{2}<X,Y>\mid v\mid^{2}.
\end{equation*}

\noindent performing the coupling  between a self-dual 2-form $F_{A}$ and a positive spinor field $v$; $ \mid\sigma(v)\mid^{2}=\frac{1}{4}\mid\ v\mid^{4}$.  The configuration space is $\ca=\Aa\times\vsa$. 

\begin{definition}
An element $(A,\phi)$ is a $\swa$-monopole if it verifies   the $\sw$-equations   ~(\ref{E:01}). There  are  two kinds of $\swa$-monopoles 
 (i) irreducible if $\phi\ne 0$ and (ii) reducible if $\phi=0$. 
\end{definition}

\noindent  The  irreducibles exist only   for a finite number  of classes in $Spin^{c}(M)$.  The monopole eqs. ~(\ref{E:01}) fits in a variational formulation whose Euler-Lagrange eqs. are the
 $2^{nd}$-order $\sw$-equations  

\begin{equation}\label{E:012}
\begin{aligned}
&d^{*}F_{A} \ +\  4iIm(<\nabla^{A}\phi,\phi>=0,\\
&\Delta_{A} \phi \ +\  \frac{\mid \phi \mid^{2}\ +\ k_{g}}{4}\phi=0.
\end{aligned}
\end{equation}

\noindent For that matter, $\jx=\{(A,0)\in\ca\mid d^{\ast}F_{A}=0\}$ is the solution set of ~(\ref{E:012}) corresponding to those connections whose curvature is harmonic,  and whose existence is guarantee by Hodge theory. Later, it will be shown that monopoles are the ground state of the theory and are also solutions of eqs ~(\ref{E:012}). In order to measure the instability, for each $\cc\in Spin^{c}(M)$,  we introduced 

\begin{equation}\label{E:0123}
\bar{\lambda}^{\cc}(M)=\sup_{A\in\jx}\left\{ \sup_{g\in\mathscr{M}}\lambda^{\cc}_{g}(A).[vol(M,g)]^{1/2} \right\} 
\end{equation}

\noindent where $\mathscr{M}$ is the space of riemannian metrics on $M$. Thus, $\jx$ is defined to be unstable if $\bar{\lambda}^{\cc}(M)<0$.

\begin{theorem}\label{T:01}
Assume $k_{g}$ is not non-negative. If there exists an irreducible solution $(A,\phi)$ of eqs ~(\ref{E:012}), then $\jx$ is unstable.
\end{theorem}

\begin{theorem}\label{T:02}
If $\cc\in Spin^{c}(M)$ admits a parallel spinor and the Yamabe invariant satisfies $Y(M)<0$, then $\jx$ is unstable.
\end{theorem}

\begin{theorem}\label{T:03}
If $\cc\in Spin^{c}(M)$  is a basic class admitting a parallel spinor, and $\alpha^{2}_{\cc}>0$, then $\bar{\lambda}^{\cc}(M)<-\pi\sqrt{\alpha^{2}_{\cc}}$. 
\end{theorem}

A class $\cc\in Spin^{c}(M)$ admitting a parallel spinor imposes strong restriction on $M$ ~(\cite{AM97},~\cite{Fr}). Assuming $\pi_{1}(M)=0$ and $M$ being irreducible as cartesian product, it turn out that either $M$ is K\"{a}hler or $M$ is spin Ricci-flat.  The former  case is characterized  by the surjectivity of the Ricci tensor and the existence of an integrable complex structure $J$ on $M$ such that $\alpha_{\cc}=c_{1}(J)$ or $-c_{1}(J)$. In the last,  the Ricci tensor must be null and the manifold spin.  The author is not  aware of any sort of classification theorem of spin Ricci-flat 4-manifolds, but its  importance for physicists.
It is a long standing problem to find  examples of Ricci-flat manifolds with holonomy $SO_{n}$.


\section{Background}
Consider $\pi:E\rightarrow M$ a vector bundle with structural group $G$ and denote $F(E)$  the $G$-principal bundle of frames on $E$.  

\subsection{Gauge Group}

Consider $G$ a Lie group with Lie algebra $\mathfrak{g}$. The Gauge group  $\G_{P}$ of a principle $G$-bundle $P_{G}$ is the set of $G$-equivariant automorphism $\Phi:P_{G}\rightarrow P_{G}$ such that $\pi\circ\Phi=\pi$. A gauge transformation $\Phi$ is better described as  a map $s:P_{G}\rightarrow G$ such that $\Phi(p)=p.s(p)$, $s(p.g)=g^{-1}.s(p).g$.  Taking the adjoint action $Ad_{g}:G\rightarrow G$, $Ad_{g}(x)=g^{-1}.x.g$, and defining the bundle 
$Ad(G)=P_{G}\times_{Ad} G$, the gauge group $\G$ is the space of sections of $Ad(P)$.  The representation $ad:G\rightarrow End(\mathfrak{g})$,   $ad(g)=g^{-1}vg$,  induces the 
associated vector bundle  $ad(\mathfrak{g})=P_{G}\times_{ad}\mathfrak{g}$. If $G$ is abelian, then $Ad(P)=Map(M,G)$; e.g.: $G=U_{1}$, $\G=Map(M,U_{1})$ and $ad(\mathfrak{u}_{1})=i\R$. 
The group $\G_{P}$ also acts on an associated vector bundle $E=P\times_{\rho} V$, $\rho:G\rightarrow End(V)$. The homotopy type of $G_{P}$ depends on the homotopy type of $P$.

\subsection{Spin and Spin$^{c}$ Structures on $M$} 

 Whenever $M$ admits a spin structure or a $\spinc$  structure it carries a Dirac operator  useful to study geometric and topological properties  on $M$ by analytical methods. 
  In order to define such structures we consider the Lie groups $Spin_{4}=SU_{2}\times SU_{2}$, recalling that $\overline{Ad}:Spin_{4}\rightarrow SO_{4}$ is the universal covering map,  and 
  $Spin^{c}_{4}=Spin_{4}\times_{\Z_{2}} U_{1}$.  Let $\pi:F(M)\rightarrow M$ be the frame bundle of $M$. A spin structure on $M$ is 
a  principal $Spin_{4}$-bundle $P^{\s}$ such that the projection  $\pi':P^{\s}\rightarrow M$ lifts to a map 
  $\zeta:P^{\s}\rightarrow F(M)$ satisfying the following conditions \\
  (i) $\zeta(p.g)=\zeta(p).\overline{Ad}(g)$, for all $p\in P^{\s}(E)$ and $g\in Spin_{4}$, \\
  (ii) $\pi\circ \zeta=\pi':P^{\s}\rightarrow M$ 
    
\noindent  It turns out that  $M$ admits a spin structure if and only if $w_{2}(M)=0$; in this case  the space of $spin$-structures on $M$ is $Spin(M)=H^{1}(M,\Z_{2})$. All $spin$ structure on a smooth 4-manifold $M$  carries a spin vector bundle $\svb=P^{\s}\times_{\rho_{\s}} \mathbb{H}^{2}$ over $M$, whose fibers are  a $Cl_{4}$-module ($Cl_{4}$ is the real Clifford Algebra isomorphic to $M_{2}(\mathbb{H})$).  From the representation theory of Clifford Algebras, there exist a decomposition $\svb=\svb^{+}\oplus\svb^{-}$ induced by inequivalent representations $\rho_{\pm}:Spin_{4}\rightarrow\mathbb{H}$. 
 In general, $M$ may not admit a $spin$ structure  because $w_{2}(M)\ne0$, but it always admits a $\spinc$ structure because there exists a class $\alpha\in H^{2}(M,\Z)$ such that $w_{2}(X)\equiv \alpha\ mod\ 2$.  Indeed, a $\spinc$-structure on $M$ corresponds to define an almost complex structure on $M\backslash {\{pt\}}$. When  $M$ is  a spin manifold we have  $c_{1}(\la)=\alpha_{\cc}\in H^{2}(M,2\Z)$. In this case,  the bundles $\svb$ and $\la^{1/2}$ are globally defined and 
$\csa=\svb\otimes(\la)^{1/2}$, where  $(\la)^{1/2}$ is the square  root bundle of $\la$.  When $w_{2}(M)\ne 0$ the tensor product $\csa=\svb\otimes(\la)^{1/2}$ is globally defined, though
 the bundles  $\svb$ and $(\la)^{1/2}$ are not. The bundle $\csa$ inherits the decomposition  $\csa=\csa^{+}\oplus\csa^{-}$, where $\csa^{\pm}$ are the $(\pm)$-complex spinor bundles of rank $2$. 
Moreover,

\begin{equation*}
\begin{aligned}
c_{1}(\csa^{+})&=c_{1}(\la),\ c_{2}(\csa^{+})=\frac{1}{4}[c^{2}_{1}(\la)-2\chi(M)-3\sigma(M)],\\
c_{1}(\csa^{-})&=c_{1}(\la),\ c_{2}(\csa^{-})=\frac{1}{4}[c^{2}_{1}(\la)+ 2\chi(M)-3\sigma(M)].
\end{aligned}
\end{equation*}

\section{Geometric Structures}

A brief introduction on covariant derivatives and curvature is given in order  to fix the concepts and the notations needed along the text.

 \subsection{Covariant Derivatives and Connections 1-forms on $\csa^{\pm}$}
 
Let's consider the general case of a smooth vector bundle $E$ over $M$. Let $\mathcal{A}_{E}$ be the space of connection 1-forms on $E$.  A covariant derivative on  a vector bundle $E$ over $M$ is a $\R$-linear operator $\nabla:\Omega^{0}(E)\rightarrow\Omega^{1}(E)$ satisfying the Leibnitz rule:  for all $f:M\rightarrow\R$ and $ V\in\Omega^{0}(E)$,
 
 \begin{equation*}
 \nabla(fV)=df\wedge V+f\wedge\nabla V.
 \end{equation*}
 
\noindent Using the exterior derivative $d:\Omega^{p}(M)\rightarrow\Omega^{p+1}$, it can be extended to a linear operator $d^{\nabla}:\Omega^{p}(E)\rightarrow\Omega^{p+1}(E)$, by

\begin{equation*}
d^{\nabla}(V\omega)=\nabla V\wedge\omega + V\otimes d\omega.
\end{equation*}

\noindent  $\mathcal{A}_{E}$   is an afim space which turns out to be  the vector space  $\Omega^{1}(ad(\mathfrak{g}))$ by fixing an origin at $\nabla^{0}\in\mathcal{A}_{E}$.
 Any covariant derivative $\nabla^{A}$ can be written as  $\nabla^{A}=\nabla^{0}+A$, $A\in \Omega^{1}(ad(\mathfrak{g}))\overset{loc.}{=}\Omega^{1}(M)\otimes \mathfrak{g}$. Covariant derivatives and connection 1-forms are equivalent.
 The  group $\G$ acts on $\mathcal{A}_{E}$ by $g.\nabla= g^{-1}\nabla g$, so inducing on $\Omega^{1}(ad(\mathfrak{g}))$  the $\G$-action $g.A=g^{-1}Ag+g^{-1}dg$.
  Fix a local chart $U\subset M$, let $\nabla:\Omega^{0}(TM)\rightarrow \Omega^{1}(TM)$ be the riemannian connection on $M$ and   $\beta_{U}=\{e_{i}\mid 1\le i\le 4\}$  be a  local orthonormal frame 
  of $TM$ defined on $U$ having the following properties: for all $i,j$ ($\nabla_{i}=\nabla_{e_{i}}$)\\
(i) $[e_{i},e_{j}]=\nabla_{i}e_{j}-\nabla_{j}e_{i}=0$, \\
(ii) $\nabla_{i}e_{k}=\sum_{l}\Gamma^{l}_{ik}e_{l}$,\\



\noindent  The covariant derivative operator is locally given by    $\nabla=\sum_{i}(\nabla_{i})dx^{i}=d+\Gamma$, where 
$\nabla_{i}=\partial_{i}+\Gamma_{i}$  and $\Gamma=\sum_{i}\Gamma_{i}dx^{i}\in \Omega^{1}(ad(\mathfrak{so}_{4}))$ is the connection 1-form. The set of  linear maps 
 $e_{k}\wedge e_{l}:\R^{4}\rightarrow\R^{4}$, given by
  
  \begin{equation*}
(e_{k}\wedge e_{l})(v)=<v,e_{l}>e_{k}-<v,e_{k}>e_{l}, \quad (\mathfrak{so}_{4}\simeq \Lambda^{2}(\R^{4})),
\end{equation*}

\noindent    defines a   $\mathfrak{so}_{4}$ basis on which  $\Gamma_{i}=\sum_{k,l}(\Gamma_{i})_{kl}(e_{k}\wedge e_{l})$.
The riemannian connection on $(M,g)$ induces a connection on $\svb$ as we shown next. Let $\C l(M,g)$ be the Clifford Algebra Bundle and $\cl:TM\rightarrow \C l(M,g)$ be the Clifford map perfoming the inclusion.
 A $\Omega^{0}(\C l(M,g))$-module structure is defined on  $\Omega^{0}(\svb)$ by the pointwise product $(\gamma.\phi)(x)=\gamma(x).\phi(x)$, for all  $\gamma\in\Omega^{0}(\C l(M,g))$ and $\phi\in\Omega^{0}(\svb)$.  In order to describe  a connection (locally defined) on $\svb$ let's consider  $\gamma_{i}=\cl(e_{i})$. The whole procedure to induce the connection  on $\svb$ relies on the lie algebra isomorphism  $\Theta: \mathfrak{so}_{n}\rightarrow \mathfrak{spin}_{n}$,
$\Theta(e_{k}\wedge e_{l})=\frac{1}{2}\gamma_{k}. \gamma_{l}$ ~(\cite{LM}, prop 6.1).
  In this way, the Christoffel's symbols of $M$   induce on $\svb$ 
  the operator $\Gamma^{\s}_{i}:\Omega^{0}(\svb)\rightarrow \Omega^{0}(\svb)$, $\Gamma^{\s}_{i}=\frac{1}{2}\sum_{l,k}\Gamma^{k}_{il}(\gamma_{k}.\gamma_{l})$.
The spin connection 1-form on $\svb$ is  $\Gamma^{\s}=\sum_{i}\Gamma^{\s}_{i}dx^{i}$, it induces  the covariant derivative $\nabla^{\s}:\Omega^{0}(\svb)\rightarrow\Omega^{1}(\svb)$,
$\nabla^{\s}=d+\Gamma^{\s}$. 
  A covariant derivative operator $\nabla^{A}:\Omega^{0}(\csa)\rightarrow \Omega^{1}(\csa)$
is locally defined  on $\csa\overset{loc.}{=}\svb\otimes \la^{1/2}$   by taking the spin connection $\nabla^{\s}$ on $\svb$ and a $U_{1}$-connection $\nabla^{A}$
 on $\la^{1/2}$,  as follows: let $\psi\in\Omega^{0}(\svb)$, $\lambda\in\Omega^{0}(\la^{1/2})$ and 
$\psi\otimes\lambda\in \Omega^{0}(\csa)$,  

\begin{equation}\label{E:471}
\nabla^{A}(\psi\otimes\lambda)=\nabla^{\s}\psi\otimes\lambda+\psi\otimes\nabla^{A}\lambda.
\end{equation}

\vspace{05pt}

\noindent (it can be patched together to define the operator $\nabla^{A}$ globally).

\subsection{Curvature}  
 
The curvature of a covariant derivative $\nabla$ on $E$ is the $C^{\infty}$-linear operator 
$F=d^{\nabla}\circ d^{\nabla}:\Omega^{0}(E)\rightarrow\Omega^{2}(E)$ defined by 

\begin{equation*}
F(V)(X,Y)=\left(\nabla_{X}\nabla_{Y} -   \nabla_{Y}\nabla_{X} - \nabla_{[Y,X]}\right)V,
\end{equation*}

\noindent for all $V\in\Omega^{0}(E)$ and $X,Y\in\Omega^{0}(TM)$. In a orthonormal frame $\beta^{E}=\{f_{\alpha}\mid 1\le \alpha\le r\}$ on $E$, such that 
$\nabla_{e_{i}}f_{\alpha}=\sum_{\beta}A^{\beta}_{i\alpha}f_{\beta}$, we consider, for each $i,j$, $F_{ij}\in End(E)$ as the operator

\begin{equation*}\label{E:120}
\begin{aligned}
F_{ij}(V)&=F(e_{i},e_{j})(V)= \left(\frac{\partial A_{j}}{\partial x_{i}}-\frac{\partial A_{i}}{\partial x_{j}}+[A_{i},A_{j}]\right)(V), 
\end{aligned}
\end{equation*}

\vspace{05pt}

\noindent 
 Each $A_{i}(x)$ being  a skew symmetric operator $\forall x$ implies $F_{ij}(x0\in End(E_{x})$ is also skew-symmetric.  
  Let $A\in\Omega^{1}(E)$  and $\nabla=d+A$, the curvature 2-forms $F_{A}\in\Omega^{2}(E)$ is $F_{A}=\sum_{i,j}F_{ij}dx^{i}\wedge dx^{j}$.
The gauge group  action on $\Omega^{2}(E)$ is $g.\omega=g^{-1}.\omega.g$ motivated by the fact that curvature of $g.A$ is $g.F_{A}=g^{-1}.F_{A} .g$.
 When $E=TM$,  the  curvature 2-form $R:\Omega^{0}(TM)\rightarrow\Omega^{2}(TM)$ of the riemannian metric is locally written, using the frame $\beta_{U}$, as $R=\sum_{i,j}R_{ij}dx^{i}\wedge dx^{j}$. The components  $R_{ij}(e_{k})=\sum_{l}R^{l}_{ijk}e_{l},\ (R_{ij})_{lk}=R^{l}_{ijk}$ satisfy the identities 

\begin{equation}
\begin{aligned}
&(i)\ R^{l}_{ijk}+R^{l}_{jki}+R^{l}_{kij}=0,\\
&(ii) R^{l}_{ijk}=-R^{l}_{jik}
\end{aligned}\quad
\begin{aligned}
&(iii)\ R^{l}_{ijk}=-R^{k}_{ijl}\\
&(iv)\ R^{l}_{ijk}=R^{j}_{kli}
\end{aligned}
\end{equation}

\vspace{05pt}

\noindent 
Using  the $\mathfrak{so}_{4}$ basis $\{e_{k}\wedge e_{l}\}$ we have $R_{ij}=\sum_{k,l}R^{k}_{ijl}e_{k}\wedge e_{l}$.
 In this way, the curvature 2-form induced on $\svb$ by the riemannian connection on $TM$ is 

\begin{equation*}
R^{\s}=\frac{1}{2}\sum_{k,l}\left(\sum_{i,j}R^{k}_{ijl}dx^{i}\wedge dx^{j}\right)\gamma_{k}.\gamma_{l}\in \Omega^{2}(\svb)
\end{equation*}

\begin{definition}
The Ricci curvature of the riemannian manifold $(M,g)$  is the bilinear form $Ric:\Omega^{0}(TM)\times\Omega^{0}(TM)\rightarrow C^{\infty}(M)$

$$Ric(u,v)=trace_{g}\left[w\to R(u,w)v\right].$$

\vspace{05pt}

\end{definition}

\noindent  Using the frame $\beta_{U}=\{e_{j}\}$ on M,  the Ricci curvature is given by 

\begin{equation*}
Ric(u,v)=g( \sum_{k}R(e_{k},v)e_{k},u), \ \text{$\forall u,v\in TM$.}
\end{equation*}

\noindent  Using the symmetry $Ric_{ij}=Ric_{ji}$, we define  the linear  self-adjoint Ricci operator 
 $Ric:\Omega^{0}(TM)\rightarrow\Omega^{0}(TM)$,  $Ric(u,v)=g(u,Ric(v))$, locally given by  $Ric(v)=\sum_{k}R(e_{k},v)e_{k}$.
 It induces on $\svb$ the operator

 \begin{equation*}
Ric^{\s}(v)=\sum_{k}R(e_{k},v)\gamma_{k}
\end{equation*}

 \noindent So far, it has been showed how the riemannian connection  induces  a connection on $\svb$.  The equation ~(\ref{E:471})   induces a connection on $\csa$ 
 whose curvature 2-form $F_{A}:\Omega^{0}(\csa)\rightarrow\Omega^{2}(\csa)$ locally decomposes into

\begin{equation}\label{E:132}
F_{A}=R^{\s}+if_{A},\quad f_{A}\in\Omega^{2}(M)
\end{equation}

\vspace{05pt}

\noindent The expression ~(\ref{E:132}) reflects the existence of the decomposition $\mathfrak{u}_{2}=\mathfrak{su}_{2}\oplus\mathfrak{u}_{1}$. By projecting the curvature $F_{A}\in\mathfrak{u}_{2}$   
into the sub-algebras the component $\mathfrak{su}_{2}$ gives part of the Riemann tensor of $M$ and the  $\mathfrak{u}_{1}$ component  gives the curvature of $A$ on $\la^{1/2}$. Thus, the curvature induced on $\la$ is $2if_{A}$.


\section{Variational Formulation and $2^{nd}$ Variation}


By fixing  an origin at $\nabla^{0}\in \Aa$  a connection on $\la$ is written as $\nabla^{A}=\nabla^{0}+A$, where $A\in \Omega^{1}(M,i\R)$ is a $\mathfrak{u}_{1}$-valued 1-forms. 
A topology on the configuration space  $\ca=\Aa\times\vsa$  is defined by considering the Sobolev spaces 
  $\Aa= L^{1,2}(\Omega^{1}(M,i\R))$ and $\varGamma(\csa^{+})$ = $L^{1,2}(\varOmega^{0}(X,\csa^{+}))$; the gauge group  is taken to be $\G=L^{2,2}(Map(X,U_{1}))$.
The  $\G$ action on $\ca$   is not free, the isotropy group
are $G_{(A,0)}=\{I\}$ and $\G_{(A,0)}\simeq U_{1}$ for all $A\in \Aa$. An element $(A,\phi)\in\ca$ is named irreducible if $\phi\ne 0$, otherwise is  reducible.
The  subspace of irreducibles $\ca^{\ast}=\{(A,\phi)\in\ca\mid \phi\ne 0\}$ is a universal principal $\G$-bundle over  the moduli space $\B^{\ast}=\ca^{\ast}/\G$.  The quotient space $\B^{\ast}$  has the same homotopy type of  $\C P^{\infty}\times\jx$. The free action of $U_{1}=\{g\in\G\mid \text{g constant}\}$  on $\ca^{\ast}$ defines a principal $U_{1}$-bundle over $\B^{\ast}$ whose first Chern class $c_{1}(\ca^{\ast})=SW(\cc)$ is the generator of subring  corresponding to the cohomology of the  factor $\C P^{\infty}\times\{p\}$ in $H^{\ast}(\B^{\ast};\Z)$ ($p\in\jx$) . 

The riemannian structure on the tangent bundle  $T\ca=\ca\times (\Omega^{1}(i\R)\oplus\vsa)$ is  the product  of the following structures on each component;  \\
(i) on $\Aa$, for all $\eta,\theta\in\Omega^{p}(M,i\R)$,

$$<\eta,\theta>_{}=\int_{M}(\eta\wedge \ast\theta)dv_{g},$$

\noindent recalling that the Hodge operator is minus the usual star operator because the forms take values in $i\R$ instead of $\R$.

\noindent (ii) on $\vsa$, for any sections $V,W\in\vsa$, ($z\in\C$, $\mathfrak{Re}(z)=\frac{z+\bar{z}}{2}$)

\begin{equation*}
<V, W>=\int_{X}\mathfrak{Re}(<V,W>)dv_{g}.
\end{equation*}

\vspace{05pt}

\noindent Thus, the inner product $<,>:T_{(A,\phi)}\ca\times T_{(A,\phi)}\ca\rightarrow \R$ is

\begin{align*}
 <\eta+V,\theta +W>=<\eta,\theta>+<V,W>.
\end{align*}

The Seiberg-Witten equations fit into a variational set up by defining the functional $\sw:\ca\rightarrow \R$,
 
 {\small{
 
 \begin{equation}\label{E:04}
 \begin{aligned}
\swa(A,\phi) = \int_{X}\{\frac{1}{4}\mid F_{A}\mid^{2} + \mid \nabla^{A} \phi
\mid^{2} + \frac{1}{8}(\mid\phi\mid^{2}+k_{g})^{2}-\frac{k^{2}_{g}}{8}  \}dv_{g} +2\pi^{2}N_{\cc},
\end{aligned}
\end{equation}
}}
\vspace{05pt}

\noindent   where $k_{g}$ is the scalar curvature of $(X,g)$ and 

$$N_{\cc}=c^{2}_{1}(\cc)=c_{1}(\cc)\wedge c_{1}(\cc)=\frac{1}{4\pi^{2}}\int_{X}[\mid F^{+}_{A}\mid^{2}-\mid F^{-}_{A}\mid^{2}]dv_{g}.$$ 

\noindent The functional $\swa$ is  gauge invariant,  therefore it is well defined on $\B$.  Defining $\bar{k}_{g}=\max\{0,\sup_{x\in M}(-k_{g}(x))\}$, it is straightforward from eq. ~(\ref{E:04}) that a necessary condition to the existence of an irreducible  monopole is $\mid\mid \phi\mid\mid_{\infty}<\sqrt{\bar{k}_{g}}$.  Jost-Peng-Wang proved  in ~\cite{JPW96}  the functional $\swa:\B\rightarrow\R$ satisfies the Palais-Smale condition, so the critical sets are compact and the minimum is always achieved ($\swa\ge 0$).
  Let $grad(\swa)(A,\phi)$ be the gradient at $(A,\phi)$, the Euler-Lagrange equations defined by $grad(\swa)(A,\phi)=0$ are ($Im(z)=\frac{z-\bar{z}}{2i}$)

\begin{equation}\label{E:123}
d^{*}F_{A} \ +\  4iIm(<\nabla^{A}\phi,\phi>=0,\quad \Delta_{A} \phi \ +\  \frac{\mid \phi \mid^{2}\ +\ k_{g}}{4}\phi=0.
\end{equation}

\vspace{05pt}

 \noindent These are the  $\G$-invariant $2^{nd}$-order $SW$-equations because  
 
 $$grad(\swa)(g.(A,\phi))=g^{-1}.grad(\swa)(A,\phi).$$
 
 \noindent  They may admit irreducible and reducible solutions. As expected, the $SW$-monopoles    also satisfy   eqs ~(\ref{E:123}),  as asserted by  the identities
 
 \begin{equation*}
\begin{aligned}
&d^{\ast}(F_{A}) =2d^{\ast}F^{+}_{A}-=d^{\ast}[\sigma(\phi)]=-4i Im\left(<D^{+}_{A}\phi,X.\phi>+<\nabla^{A}_{X}\phi,\phi>\right)\\
&D^{+}_{A}\phi =0\ \Rightarrow\ 0=D^{-}_{A}D^{+}_{A}\phi=\triangle_{A}+\frac{k_{g}}{4}\phi+\frac{F^{+}_{A}}{2}.\phi=\triangle_{A}\phi+\frac{k_{g}}{4}\phi+\frac{\mid\phi\mid^{2}}{2}\phi
\end{aligned}
\end{equation*}

\noindent  Due to  the identity $\swa(A,\phi)=\int_{M}(\mid F^{+}_{A}-\sigma(\phi)\mid^{2}+\mid D^{+}_{A}\phi\mid^{2})dv_{g}$,  whenever $\cc\in Spin^{c}(X)$ is a basic class the  $\sw$-monopoles
  are stable critical points.    The solution set of eqs. ~(\ref{E:123}) may be singular in the presence of reducible points. 
 Assuming $k_{g}\ge 0$, the minimum is achieved at reducible points because $\swa(A,\phi)\ge \swa(A,0)$, $\forall (A,\phi)\in\ca$. At a critical point $(A,0)$,  the $\swa$-monopole eqs ~(\ref{E:01}) is $F^{+}_{A}=0$ and the eqs. ~(\ref{E:123}) reduces to $d^{\ast}F_{A}=0$. 
Under the assumption  $b^{+}_{2}(M)\ge 2$, independently of the sign of $k_{g}$,   the anti-self-dual solutions can be ruled out. 
If $d^{\ast}F_{A}=0$, then  $F_{A}$ is  harmonic  2-form. Let $B\in\Aa$ be such that  $F_{A}=F_{B}$, so $\omega=B-A\in H^{1}(M,\R)$. Moreover, $B$ is gauge equivalent to $A$ 
 if and only if $\omega=g^{-1}dg\in H^{1}(M,\Z)$. Therefore,  the space
 $\jx=\{(A,0)\in\ca\mid d^{\ast}F_{A}=0\}\diagup\G$   is diffeomorphic to the jacobian torus $T^{b_{1}(X)}=\frac{H^{1}(M,\R)}{H^{1}(M,\Z)}$. 
A local slice of $\B$ at $(A,\phi)$ is given  ~(\cite{JM96})  by the kernel $ker(T^{\ast}_{\phi})=ker(d^{\ast})\oplus \phi^{\perp}$ of the operator 

\begin{equation}
\begin{aligned}
T^{*}_{\phi}:\Omega&^{1}(X,i\R)\oplus\vsa\rightarrow \Omega^{0}(X,i\R),\\
&T^{*}_{\phi}(\theta,V)=d^{*}\theta\ -<V,\phi>,.   
\end{aligned}
\end{equation}

\vspace{05pt}

\noindent  Because
$(d^{\ast})^{2}=0$, it decomposed
 into  subspaces $ker(d^{\ast})= d^{\ast}(\Omega^{2}(M,i\R))\oplus\mathcal{H}_{1}$, where the subspace of harmonic 1-forms $\mathcal{H}_{1}=\{\theta\in\Omega^{1}(M,i\R)\mid d\theta=d^{\ast}\theta=0\}$ is the tangent space to the Jacobian torus $\jx$ at $(A,0)$. 
 The instability of $\jx$ is established by performing the analysis of the $2^{nd}$ variation $\frac{\delta^{2}\sw}{\delta\alpha\delta\beta}$ of the $\sw$-functional. 
The tangent space of $\ca$ at $(A,\phi)$ is $T_{(A,\phi)}\ca=\Omega^{1}(X;i\R)\oplus\vsa$, so
$\frac{\delta^{2}\sw}{\delta\alpha\delta\beta}$  defines a symmetrical bilinear form $\h((\theta_{1},V_{1}),(\theta_{2},V_{2}))=<(\theta_{1},V_{1}),H(\theta_{2},V_{2})>$, 
where the operator $H=\left(
\begin{matrix}
h_{11} & h_{12}\\
h_{21} & h_{22}
\end{matrix}\right)$ has entries given by 
 
 \begin{equation*}
\begin{aligned}
\frac{\delta^{2}\swa}{\delta\Lambda\delta\theta}\mid_{(A,\phi)}.(\theta,\Lambda)&=<\theta,(d^{\ast}d\Lambda+4<\Lambda(\phi),\phi>)=<\theta,h_{11}(\Lambda)>,\\
\frac{\delta^{2}\swa}{\delta W\delta\theta}\mid_{(A,\phi)}.(\theta,W)&=2\left(<\nabla^{A}\phi,\theta(W)> + <\nabla^{A}W,\theta(\phi)>\right)=\\
&=<\theta,h_{12}(W)>,\ (h_{21}=h_{12})\\
\frac{\delta^{2}\swa}{\delta W\delta V}\mid_{(A,\phi)}.(V,W)&= <V,\triangle_{A}W + \frac{k_{g} + \mid\phi\mid^{2}}{4}W+\frac{1}{4}<\phi,W>\phi>=\\
&=<V,h_{22}(W)>.
\end{aligned}
\end{equation*}
\vspace{05pt}

\noindent The  restriction of the $2^{nd}$-variation   to the slice of $\B$ at $(A,\phi)$  is an elliptic operator   
$H:ker(T^{\ast}_{\phi})\rightarrow ker(T^{\ast}_{\phi})$ whose leading terms  $d^{\ast}d=\triangle$  and  $\triangle_{A}=-(\nabla^{A})^{\ast}\nabla^{A}$ are laplacians and whose  tail is a compact operator. Thus, $H$ is a self-adjoint Fredholm operator.  The spectrum $\sigma(H)$
  is a discrete set such that each eigenvalue has finite multiplicity and no accumulation points, besides,  there are but a finite number of eigenvalues below any given number.
 At  $(A,0)$, the hessian 
 operator becomes
  $ H = \left(
 \begin{matrix}
 d^{*}d  & 0  \\
 0 & L_{A}
 \end{matrix}
 \right)$, 
 where $L_{A}:\vsa\rightarrow\vsa$ is the elliptic self-adjoint operator  
 
 \begin{equation}\label{E:456}
L_{A}(V)=\triangle_{A}V+\frac{k_{g}}{4}V.
\end{equation}

 \noindent   For each $\lambda\in\sigma(L_{A})$, the  corresponding eingenspace $\mathcal{V}_{\lambda}\subset T_{(A,0)}\B$ 
   has finite dimension. In this way, $ker(\mathcal{H})=T_{(A,0)}\mathcal{J}_{X}\oplus \mathcal{V}_{0}$.
   The  lower eigenvalue of $L_{A}$  given by Rayleigh's quotient

\begin{equation}
\begin{aligned}
\lambda^{\cc}_{g}(A)&=\inf_{V\in\csa^{+}} \frac{\int_{M}\{\mid \nabla^{A}V\mid^{2}\ + \ \frac{k_{g}}{4}\mid V\mid^{2}\}dv_{g}}{\int_{M}\mid V\mid^{2}dv_{g}}
\end{aligned}
\end{equation}

\noindent is bounded below, so is the spectrum $\sigma(L_{A})$.


\subsection{Parallel Spinor}

A spinor $\psi\in\vsa$ is parallel with respect to a connection $\nabla$ if  $\nabla \psi=0$.  In general, it is  difficult to pull off information about  $\sigma(L_{A})$, 
but using Kato's inequality  it can be compare  with the spectrum of $L=\triangle_{g}+\frac{k_{g}}{4}$  defined on functions $f:M\rightarrow\R$ ($\triangle_{g}$=Laplace-Beltrami).
Consider on $M$ a smooth atlas  $\mathcal{A}(M)=\{(U_{\lambda},\xi_{\lambda})\mid\lambda\in \lambda\}$ such that,  for each $\lambda\in\Lambda$,   (i) $U_{\lambda}$ is convex, (ii)
 the local coordinates are $\{(x_{1},x_{2},x_{3},x_{4})\in U_{\lambda}\mid x_{i}\in\R\}$ and (iii) attached to $U_{\lambda}$ there exists a local orthonormal frame 
$\beta_{\lambda}=\{e_{i}\mid e_{i}=\partial_{i},\ 1\le i\le 4\}$.

\begin{proposition}\label{P:234}
(Kato's ineq.) Let $A\in\Aa$ and $V\in\vsa$. Then,

\begin{equation}\label{E:90}
\mid\nabla\mid V\mid^{2}\mid\ \le\  \mid \nabla^{A}V\mid^{2}
\end{equation}

\noindent The equality holds if, and only if, there exists a 1-form $\omega\in\Omega^{1}(M)$ such that  $\nabla^{A}V=\omega V$.
\end{proposition}
\begin{proof}
Taking the orthonormal frame $\beta=\{e_{i}\mid 1\le i\le 4\}$, locally we get $\mid\nabla\mid V\mid\mid^{2}=\sum_{i}\mid\nabla_{i}\mid V\mid\mid^{2}$ and 
$\mid\nabla^{A} V\mid^{2}=\sum_{i}\mid\nabla^{A}_{i}V \mid^{2}$.  From the identities $\nabla_{i}\mid V\mid^{2}=2\mid V\mid.\nabla_{i}\mid V\mid$ and 
$\nabla_{i}\mid V\mid^{2}=\nabla_{i}<V,V>=2<\nabla^{A}V,V>$, we have $\mid V\mid.\nabla_{i}\mid V\mid=<\nabla^{A}_{i}V,V>$.
Assuming $V\ne 0$ and applying Cauchy-Scwartz inequality it follows the inequality $\mid\nabla_{i}\mid V\mid\mid \ \le \ \mid\nabla^{A}_{i}V\mid$.
 Hence, ineq. ~(\ref{E:90}) is verified. The equality is attained whenever there exists functions $\alpha_{i}:M\rightarrow\C$ such that $\nabla^{A}_{i}V=\alpha_{i}V$, that is,

\begin{equation*}
\nabla^{A} V=\sum_{i}\nabla^{A}_{i}Vdx^{i}=\left[\sum_{i}\alpha_{i} dx^{i}\right] V=\omega V
\end{equation*}

\end{proof}

\noindent If  $V$ is a harmonic spinor ($D_{A}V=0$) and $\nabla^{A}V=\omega V$, then $\nabla^{A}V=0$. It is rather restrictive to assume  $V$ as a harmonic spinor, but under  an extra assumption on the functions $\alpha_{i}:M\rightarrow\C$ the existence of a parallel spinor can be achieved.
\noindent 
The reverse claim is also true; 

\begin{proposition}\label{P:02}
There exists a parallel spinor $V\in\vsa$ if, and only if, there exists a spinor $V_{0}\in\vsa$ and a class $\omega\in H^{1}_{dR}(M)$ such that  $\nabla^{A} V_{0}=\omega V_{0}$.
\end{proposition}
 \begin{proof}
 Suppose $V\in\vsa$ is parallel, $\nabla^{A}V=0$. So, $V$ has constant length. Let $V=fV_{0}$, where $f:M\rightarrow\C$, so $f(x)\ne 0$ and $V_{0}(x)\ne 0$, $\forall x\in M$. Furthermore, 

\begin{equation}
\nabla^{A}V=df\wedge V_{0}+f\wedge\nabla^{A}V_{0}=0\ \Rightarrow\ \nabla^{A}V_{0}=-\frac{df}{f}V_{0}.
\end{equation}

\noindent The 1-form $\omega=-\frac{df}{f}=-d(ln(f))$ is exact. Now, let's prove the reverse assuming that $\nabla V_{0}=\omega V_{0}$ and $d\omega=0$. The equation $\nabla^{A}(fV_{0})=0$ is equivalent to 
 $df-f\omega=0$; in this case  $\omega=-d(ln(f))$.
 Taking a local chart $(U_{\lambda},\phi_{\lambda})$ from the atlas defined at the beginning of this section, say   a chart $(U_{\lambda},\phi_{\lambda})$ with frame 
$\beta_{\lambda}=\{e_{i}\mid 1\le i\le 4\}$, and  defining $\alpha_{i}=w(e_{i})$,  we get  $\omega=\sum_{i}\alpha_{i}dx^{i}$. In this way,  the equation
$df-f\omega=0$ becomes locally  described by the system $\partial_{i}f-\alpha_{i}f=0$, $1\le i\le 4$.
 The  closedness of $\omega$ is equivalent to the conditions $\partial_{j}\alpha_{i}=\partial_{i}\alpha_{j}$, for all $i,j$. Thus, the necessary condition  
 $\partial_{j}\partial_{i}f=\partial_{i}\partial_{j}f$ to the existence of $f$ is easily verified, since

$$\partial_{j}\partial_{i}f=-(\partial_{j}\alpha_{i})f-\alpha_{i}\alpha_{j}f=\partial_{i}\partial_{j}f.$$

\noindent  The identity $\partial_{j}\alpha_{i}=\partial_{i}\alpha_{j}$ allow us to integrate and write

\begin{equation*}
\alpha_{i}(x_{1},x_{2},x_{3},x_{4})=\int_{0}^{x_{1}}\partial_{i}\alpha_{1}(t,x_{2},x_{3},x_{4})dt,\ 2\le i\le 4.
\end{equation*}

\noindent Therefore, the function 

\begin{equation*}
f(x_{1},x_{2},x_{3},x_{4})=e^{\int^{x_{1}}_{0}\alpha_{1}(t,x_{2},x_{3},x_{4})dt},
\end{equation*}

\noindent satisfies $\partial_{i}f-\alpha_{i}f=0$, for  $1\le i\le 4$, and is $C^{\infty}$.   The function $f$ is globally defined because it depends only on the 1-form $\omega$.
  \end{proof}

\noindent    As before,  consider   $\beta=\{e_{\alpha}; 1\le \alpha\le 4\}$ an orthornormal frame 
on $M$ and $\gamma_{\alpha}=\cl(e_{\alpha})$. The Ricci operator induces the operator $\cl(Ric(.)):TX\rightarrow Cl(X)$,
$\cl(Ric(X))=\sum_{\alpha}R^{\s}(e_{\alpha},X)\gamma_{\alpha}$,  such that 

\begin{equation*}
\begin{aligned}
\left[\cl(Ric(X))\right]^{2}=-\sum_{\alpha}\mid R^{\s}(e_{\alpha},X)\mid^{2}=-\mid Ric(X)\mid^{2}.
\end{aligned}
\end{equation*}

\vspace{05pt}

\begin{definition}
Let $A\in\Aa$ be a connection 1-form with curvature $if_{A}\in\Omega^{2}(M,i\R)$;
\begin{enumerate}
\item Let $H_{A}:TX\times \Omega^{0}(\csa)\rightarrow\Omega^{0}(\csa)$ be the linear operator  defined  by

\begin{equation*}
\begin{aligned}
&H_{A}(X,\psi)=-\frac{1}{2}\sum_{\alpha}\gamma_{\alpha}.F^{\cc}_{A}(e_{\alpha},X)(\psi),
\end{aligned}
\end{equation*}

\item Let $I_{A}:\Omega^{0}(TM)\rightarrow\Omega^{0}(TM)$ be the skew-symmetric operator

\begin{equation}
\begin{aligned}
I_{A}(X)&=\sum_{\alpha}[(X\ \llcorner\ f_{A})(e_{\alpha})]e_{\alpha}\\
\end{aligned}
\end{equation}

\end{enumerate}

\end{definition}

\noindent Using the local frame $\beta=\{e_{i}\}$ we have $(I_{A})_{\alpha\beta}=i(e_{\beta}\llcorner f_{A})(e_{\alpha})=2if_{\beta\alpha}$.  
Let $A\in\Aa$ and assume $\psi\in\Omega^{0}(\csa)$ is a parallel spinor.
Global information about $M$ can be draw, as we shall see next, upon the existence of a parallel spinor $\psi$. In this case, it follows  that  $H_{A}(X,\psi)=0$, for all $X\in TM$. The first consequence is the identity $\mid\mid Ric\mid\mid^{2}=\mid\mid f_{A}\mid\mid^{2}$ whose  proof  in ~(\cite{Fr}, chap 3) relies strongly on the  identity 

\begin{equation}
[\cl(Ric(X))-i\cl(I_{A}(X))].\psi=0 ,\ \forall X\in TM.
\end{equation}

\noindent  So, $f_{A}=0$ if and only if $Ric=0$. Moreover,  
(i) $\mid \cl(Ric(X))\mid=\mid \cl(I_{A}(X))\mid$ and (ii) $<\cl(Ric(X)),i\cl(I_{A}(X))>=0$. The identity $\cl(Ric(X))(\psi)=i\cl(I_{A})(X))(\psi)$ is a key point. Defining
 $\mathrm{R}=\{Ric(X)\mid X\in\Omega^{0}(TX)\}$, 
 from the self-adjointness of Ricci operator  there exists the decomposition $TX=\mathrm{R}\oplus ker(I_{A})$ ($ker(I_{A})=\mathrm{R}^{\perp}$).  Let's consider the operators 

\begin{equation*}
\begin{aligned}
\Psi:TX\rightarrow \Omega^{0}(\svb)\\
X\to\cl(X).\psi
\end{aligned}
\quad
\begin{aligned}
\Psi^{\iota}:TX\rightarrow \Omega^{0}(\svb)\\
X\to i\cl(X).\psi
\end{aligned}
\end{equation*}

\noindent and the vector space  $\mathrm{E}=\Psi^{-1}(Imag(\Psi)\cap Imag(\Psi^{\iota}))$. The existence of a parallel spinor $\psi$ means that $\mathrm{R}\subset \mathrm{E}$ and so 
$\mathrm{E}^{\perp}\subset ker(I_{A})$.  These spaces define the distributions $\mathcal{E}=\{\mathrm{E}_{x}\mid x\in M\}$  and $\mathcal{E}^{\perp}=\{\mathrm{E}^{\perp}_{x}\mid x\in M\} $ 

\begin{proposition}
If $\psi$ is a parallel spinor, then the distributions $\mathcal{E}$  and $\mathcal{E}^{\perp} $ are integrable. 
\end{proposition}
 \begin{proof}
  For all $X\in\mathcal{E}$ there exist an unique $Y\in TX$ such that $X.\psi=iY.\psi$. The space $E$ is closed under the 
  action of the covariant derivative because  $X.\psi=iY.\psi$ implies $(\nabla X).\psi=i(\nabla Y).\psi$.   
  
\noindent (i) $\mathcal{E}^{\perp}$ is integrable.\\
Note that for all $X\in\mathcal{E}^{\perp}$ we get $X\llcorner f_{A}=0$, in particular $f_{A}$ annihilates $X$. Since $df_{A}=0$, the distribution $\mathcal{E}^{\perp}$ is  integrable. 

\noindent (ii) $\mathcal{E}$ is integrable.\\
 Taking $X,Y\in \mathrm{E}$,  it follows from the $\nabla$-invariance of $\mathrm{E}$ that the commutator $[X,Y]=\nabla_{X}Y-\nabla_{Y}X\in\mathrm{E}$. 
  
 \end{proof}

\begin{corollary}
There exist submanifolds $M_{1},M_{2}\subset M$ such that $M_{1}$ is K\"{a}hler and $M_{2}$ is spin.
\end{corollary}
\begin{proof}
Let $M_{1}$ be the submanifold whose tangent space $T_{x}M_{1}=\mathrm{E}_{x}$. Since  for each $X\in\mathcal{E}$ there exists only one $Y$ such that $X.\psi=iY.\psi$, define the automorphism $J:TX\rightarrow TX$, $X.\psi=i J(X).\psi$. Thus, for each $x\in M$,  $J:T_{x}M_{1}\rightarrow T_{x}M_{1}$ defines a complex structure since

$$iJ^{2}(X).\psi=iJ(J(X)).\psi=J(X).\psi=-iX.\psi\ \Rightarrow\ J^{2}=-I.$$

\noindent Moreover,  $\nabla J=0$ because $J(\nabla X)=\nabla Y$ and, for all $X,Y\in\Omega^{0}(TM)$, 
  
  \begin{equation*}
  (\nabla X).\psi=[(\nabla J)X+J(\nabla X)].\psi=iJ(\nabla X).\psi\ \Rightarrow\ \nabla J=0.
  \end{equation*}
  
  
\noindent From the identity 

\begin{equation*}
[J(X).J(Y)+J(Y).J(X)].\psi=(XY+YX).\psi+2i[g(J(X),Y)+g(J(Y),X)].\psi
\end{equation*}

\noindent we get $g(J(X),J(Y)).\psi=\left\{g(X,Y)+i[g(J(X),Y)+g(J(Y),X)\right\}.\psi$,
 and so, $g(J(X),J(Y))=g(X,Y)$ and $g(J(X),Y)=-g(J(Y),X)$. Therefore,   $M_{1}$ is K\"{a}hler , and  
 $M_{2}$ is spin Ricci-flat  because it follows from $ f_{A}\mid_{M_{2}}=0$  that  $\cc\mid_{M_{2}} =0$, hence $w_{2}(M_{2})=0$, and $\mid\mid Ric\mid\mid^{2}=\mid\mid f_{A}\mid\mid^{2}=0$  on $M_{2}$.
\end{proof}

\noindent In the context of the arguments above, if $Ricc:TM\rightarrow TM$ is onto, then $M$ is K\"{a}hler.   Thus, taking the restriction $\mathcal{L}_{1}=\la\mid_{M_{1}}$, the canonical class of $(M_{1},J)$ is $\kappa_{J}=\mathcal{L}_{1}$. Assuming $\pi_{1}(M)=0$,    Morianu ~\cite{AM97} proved $\kappa_{J}$ and $-\kappa_{J}$ to be the only $\spinc$ classes on $M$ carrying parallel spinors, which are known to be the only basic class in $M$  ~(\cite{JM96}).
 Using de Rham's decomposition theorem, Moroianu concluded that a simply connected manifold carries a parallel spinor if, and only if, it is isometric to the riemannian product $M_{1}\times M_{2}$ where $M_{1}$ is K\"{a}hler and $M_{2}$ is spin Ricci-flat. Of course, if we assume $M$ is irreducible as cartesian product, then either $M$ is K\"{a}hler or $M$ is spin. 
To the best of author's knowledgment there is no example of a Ricci-flat spin manifold with holonomy $SO_{4}$ and it is not known the  classification of spin Ricci-flat  4-manifolds   beyond the one quoted in ~\cite{MI} .

 \section{Conclusion}
 
 In this section, Kato's inequality ~(\ref{E:90}) is used to compare  the lower eigenvalue $\lambda^{\cc}_{g}(A)$ of operator $L_{A}$ in ~(\ref{E:456}) with the the lower eigenvalue $\lambda_{g}$ 
 of  operator $L=\triangle_{g}+\frac{k_{g}}{4}$ acting on functions $f:M\rightarrow\R$.  By the Rayleigh's formula, each lower eigenvalue is given by
 
 \begin{equation}
\begin{aligned}
\lambda_{g}&=\inf_{f\in\Omega^{0}(M)} \frac{\int_{M}\{\mid \nabla f\mid^{2}\ + \ \frac{k_{g}}{4}\mid f\mid^{2}\}dv_{g}}{\int_{M}\mid f\mid^{2}dv_{g}}\\
\end{aligned}
\end{equation}

\vspace{05pt}

\begin{equation}\label{E:98}
\begin{aligned}
\lambda^{\cc}_{g}(A)&=\inf_{V\in\csa} \frac{\int_{M}\{\mid \nabla^{A}V\mid^{2}\ + \ \frac{k_{g}}{4}\mid V\mid^{2}\}dv_{g}}{\int_{M}\mid V\mid^{2}dv_{g}}
\end{aligned}
\end{equation}
  
\begin{definition}
The Perelman-Yamabe smooth invariant of M is  

\begin{equation}\label{E:567}
\bar{\lambda}(M)=\sup_{g}\lambda_{g}[vol(M,g)]^{1/2}
\end{equation}

\end{definition} 

\noindent Let $\mathscr{M}_{M}$ be the space of riemannian metrics on $M$ and $[g]=\{\zeta.g\mid \zeta:M\rightarrow(0,\infty)\}$ the conformal class of  $g$. 
 The Yamabe constant of $[g]$ is defined by

 \begin{equation}
Y_{[g]}=\inf_{\hat{g}\in[g]}\frac{\int_{M}k_{\hat{g}}dv_{\hat{g}}}{[vol(M,\hat{g})]^{1/2}}.
\end{equation}
 
 \vspace{05pt}
 
\noindent The condition $Y_{[g]}\le 0$ implies the existence of unique metric realizing the  Yamabe constant ~(\cite{RS84}). The smooth Yamabe invariant is defined as 
$\mathcal{Y}(M)=\sup_{[g]\subset\mathscr{M}}Y_{[g]}$.  Under the hypothesis  $\mathcal{Y}(M)\le 0$,  Akutagawa-Ishida-LeBrun  proved in ~\cite{AIL}  the identity 
 $\mathcal{Y}(M)=\bar{\lambda}(M)$.  By analogy, associated to the operator  $L_{A}$  we consider  

$$\bar{\lambda}^{\cc}(A)=\sup_{g\in\mathscr{M}}\lambda^{\cc}_{g}(A).[vol(M,g)]^{1/2}.$$

\noindent   and $\bar{\lambda}^{\cc}(M)=\sup_{A\in\jx}\bar{\lambda}^{\cc}(A)$. Assuming that $\mathcal{Y}(M)<0$ and $\cc\in Spin^{c}(X)$  is a class carrying a parallel spinor $\psi\in\vsa$, so
$ \bar{\lambda}^{\cc}(M)<0$ and $\jx$ is unstable, concluding Theorem ~\ref{T:02}.
  If an irreducible solution $(A,\phi)$ exists, it follows from the $\sw$-equations that

 \begin{equation*}
 \begin{aligned}
 \int_{X}\left[\mid\nabla^{A}\phi\mid^{2}+\frac{k_{g}}{4}\mid\phi\mid^{2}\right]dv_{g}=-\frac{1}{4}\int_{X}\mid\phi\mid^{4}dv_{g}.
 \end{aligned}
 \end{equation*}

\noindent  Consider $k_{g}$ isn't non-negative and let $(A,\phi)$ be an irreducible solution of the eqs. ~(\ref{E:123}),  so   from eq. ~(\ref{E:98}) we have

\begin{equation*}
\begin{aligned}
\lambda^{\cc}_{g}(A).\int_{M}\mid\phi\mid^{2}dv_{g}\le -\frac{1}{4}\int_{M}\mid\phi\mid^{4}dv_{g}\ \Rightarrow\ \lambda^{\cc}_{g}(A)<0.
\end{aligned}
\end{equation*}

\noindent Applying  Cauchy-Schwartz inequality we get

\begin{equation*}
\begin{aligned}
\int_{M}\mid\phi\mid^{2}dv_{g}\le [vol(M,g)]^{1/2}.\left[\int_{M}\mid\phi\mid^{4}dv_{g}\right]^{1/2},
\end{aligned}
\end{equation*}

\noindent and so

\begin{equation*}
\begin{aligned}
\bar{\lambda}^{\cc}_{g}(A).\left[\int_{M}\mid\phi\mid^{4}dv_{g}\right]^{1/2}\le \lambda^{\cc}_{g}(A)\int_{M}\mid\phi\mid^{2}dv_{g}\le -\frac{1}{4}\int_{X}\mid\phi\mid^{4}dv_{g}.
\end{aligned}
\end{equation*}

\noindent Hence,  $\bar{\lambda}^{\cc}(M)\le 0$, proving  theorem ~\ref{T:01}.  Whenever there exists   a $\swa$-monopole $(A,\phi)$, then 

$$\frac{1}{4}\int_{M}\mid\phi\mid^{4}dv_{g}=\int_{M}\mid F^{+}_{A}\mid^{2}dv_{g}=4\pi^{2}c^{2}_{1}(J)[M]+\int_{M}\mid F^{-}_{A}\mid^{2}dv_{g}$$

$$-\frac{1}{4}\left[\int_{X}\mid\phi\mid^{4}dv_{g}\right]^{1/2}=-\frac{1}{2}\left[\int_{M}\mid F^{+}_{A}\mid^{2}dv_{g}\right]^{1/2} \le -\pi\sqrt{ \alpha_{\cc}^{2}[M]}.$$

\noindent Let's consider the case $c^{2}_{1}(\la)[M]>0$, otherwise $\phi=0$. So, defining $\bar{\lambda}^{\cc}(M)=\sup_{A\in\jx}\bar{\lambda}^{\cc}(A)$, we get the upper bound in theorem
~\ref{T:03}

$$\bar{\lambda}^{\cc}(M)\le -\pi\sqrt{c^{2}_{1}(\la)[M]}.$$

\noindent Therefore, if $M$ admits a $\swa$-monopole, then $\lambda^{\cc}(M)<0$.  
 
It ought to be checked if $\bar{\lambda}^{\cc}(M)$ is a smooth invariant, for each $\cc\in Spin^{c}(X)$.


\vspace{20pt}

\noindent\emph{Universidade Federal de Santa Catarina \\
    Campus Universit\'ario , Trindade\\
               Florian\'opolis - SC , Brasil\\
               CEP: 88.040-900\\ phone:+55-48-96128385}
\vspace{10pt}

\end{document}